\documentclass[11pt]{amsart}

\usepackage{a4wide}
\usepackage{amssymb}
\usepackage{mathtools}
\usepackage{enumerate}
\usepackage{hyperref}
\usepackage{xcolor}
\usepackage{mathscinet}
\usepackage{esint}
\usepackage[nameinlink]{cleveref}

\theoremstyle{plain}
\newtheorem{theorem}{Theorem}[section]
\newtheorem{lemma}[theorem]{Lemma}
\newtheorem{corollary}[theorem]{Corollary}
\newtheorem{proposition}[theorem]{Proposition}

\theoremstyle{definition}
\newtheorem{definition}[theorem]{Definition}
\newtheorem{remark}[theorem]{Remark}

\numberwithin{equation}{section}

\DeclareMathOperator{\supp}{supp}

\makeatletter
\@namedef{subjclassname@2020}{\textup{2020} Mathematics Subject Classification}
\makeatother

\title{Nonlocal Harnack inequality in a disconnected region}

\author{Se-Chan Lee}
\address{School of Mathematics, Korea Institute for Advanced Study, 02455 Seoul, Republic of Korea}
\email{sechan@kias.re.kr}



\begin{document}

\begin{abstract}
    We establish a Harnack inequality for weak solutions of nonlocal equations in a disconnected region. The inequality compares the value of a solution on one connected component with its value on another, capturing a purely nonlocal phenomenon with no local analogue. We provide two different approaches: one based on the localized maximum principle and another on the Poisson kernel estimates.
\end{abstract}
	
\maketitle

\section{Introduction}\label{sec-introduction}


The classical Harnack inequality for harmonic functions can be formulated as follows.

\begin{theorem}\label{thm-harnack-classical}
    Let $u: B_{2r}(x_0) \subset \mathbb{R}^n \to \mathbb{R}$ be a nonnegative harmonic function. Then there exists a constant $C=C(n)$ such that
    \begin{equation*}
        \sup_{B_r(x_0)}u \leq C\inf_{B_r(x_0)}u.
    \end{equation*}
\end{theorem}

Although the proof of this inequality for harmonic functions is quite immediate from the Poisson integral formula, it implies significant consequences in the regularity theory; Liouville theorem, removable singularity theorem and H\"older regularity of harmonic functions, just to name a few. Moreover, the Harnack inequality remains valid for solutions of a wide class of partial differential equations, not only linear equations with measurable uniformly elliptic coefficients, but also quasilinear equations involving $p$-Laplacian. For divergence form operators, the De Giorgi--Nash--Moser theory~\cite{DG57, Nas58, Mos60, Mos61} provides a very flexible tool for studying the local behaviors of solutions such as local boundedness, weak Harnack inequality and Harnack inequality. In short, this theory is based on the iteration of Caccioppoli inequalities derived from choosing suitable test functions. We also refer to Serrin~\cite{Ser64} and Trudinger~\cite{Tru67} for quasilinear equations involving $p$-Laplace equations. On the other hand, Krylov--Safonov~\cite{KS79,KS80} developed a priori estimates (the so-called expansion of positivity) for both elliptic and parabolic operators in nondivergence form; see also the paper by Caffarelli~\cite{Caf89} for a similar result in the fully nonlinear setting. We finally refer to the monograph by Kassmann~\cite{Kas07} for further introduction and historical background of the Harnack inequality.

We now move our attention to the nonlocal Harnack inequality. For the fractional Laplacian operator $(-\Delta)^s$, the nonlocal counterpart of \Cref{thm-harnack-classical} can be found in the book by Landkof~\cite{Lan72} under the assumption that $u$ is nonnegative in the whole Euclidean space $\mathbb{R}^n$. The necessity of the global nonnegative condition on $u$ was pointed out by Kassmann~\cite{Kas07a}. In fact, in order to replace the global nonnegative condition with the nonnegative assumption only in a ball $B_{2r}(x_0)$, one needs to introduce the nonlocal tail term $\mathrm{Tail}(u_-; x_0, r)$ that encodes the long-range interaction of $u$; see \Cref{thm-Harnack} for the precise statement. The nonlocal De Giorgi--Nash--Moser theory has been extensively studied by many authors to establish the nonlocal Harnack inequality and local H\"older regularity. We refer to Bass--Kassmann~\cite{BK05a, BK05b} and Kassmann~\cite{Kas09, Kas11} for linear nonlocal operators in divergence form, Di Castro--Kuusi--Palatucci~\cite{DCKP14, DCKP16} and Cozzi~\cite{Coz17} for nonlinear nonlocal operators in divergence form, and Caffarelli--Silvestre~\cite{CS09, CS11} for fully nonlinear nonlocal operators in nondivergence form. See also \cite{BKO23, CKW23, IMS16, KW24a, KLL25, Lin16} for Harnack inequality and related results in different settings.\newline


The goal of this paper is to establish the nonlocal Harnack inequality in a \emph{disconnected region}, which becomes possible due to the nonlocal nature of equations. To illustrate the issue, we first consider the corresponding situation in the local case. Let $u$ be harmonic in $B_{2r}(x_1) \cup B_{2r}(x_2)$, where $B_{2r}(x_1)$ and $B_{2r}(x_2)$ are two disjoint balls in $\mathbb{R}^n$. Moreover, suppose that $u$ is nonnegative in a large ball containing both  $B_{2r}(x_1)$ and $B_{2r}(x_2)$. The classical Harnack inequality (\Cref{thm-harnack-classical}) says that the ratio $u(y_1)/u(y_2)$ is bounded above and below by positive universal constants, whenever two points $y_1$ and $y_2$ are contained in the same ball (either $B_{r}(x_1)$ or $B_{r}(x_2)$). However, this cannot happen when one chooses $y_1 \in B_{r}(x_1)$ and $y_2 \in B_{r}(x_2)$; just consider $u \equiv M$ in $B_{2r}(x_1)$ and $u \equiv 1$ in $B_{2r}(x_2)$ for any $M \geq 0$. In other words, the information on one connected component cannot be transported to another.

Surprisingly, the ratio between two values of a solution $u$ in disjoint balls admits a universal bound when it comes to the nonlocal equation. To be precise, let $u$ be a weak solution of linear nonlocal equations in divergence form given by
\begin{equation}\label{eq-main}
        0=\mathcal{L}u(x)\coloneqq 2\,\mathrm{p.v.}\int_{\mathbb{R}^n} (u(x)-u(y))k(x, y)\,\mathrm{d}y,
\end{equation}
where the kernel $k :\mathbb{R}^n \times \mathbb{R}^n \to [0, \infty]$ satisfies a symmetry and uniform ellipticity condition of the form
\begin{equation*}
    \Lambda^{-1}|x-y|^{-n-2s} \leq k(x, y)=k(y, x) \leq \Lambda |x-y|^{-n-2s}
\end{equation*}
for some $s \in (0,1)$ and $\Lambda \geq 1$. In particular, the choice $k(x, y)=|x-y|^{-n-2s}$ (up to a normalization constant) corresponds to the fractional Laplacian operator $(-\Delta)^s$. Moreover, we fix notation on a disconnected region consisting of two disjoint balls throughout the paper. Let $x_1, x_2 \in \mathbb{R}^n$ be points such that $4r\leq |x_1-x_2| \leq 8r$. Suppose that $B_{2r}(x_1) \cup B_{2r}(x_2) \subset B_{R/2}=B_{R/2}(0)$. In particular, these assumptions imply that $B_{2r}(x_1) \cap B_{2r}(x_2)=\varnothing$ and that $|x_1|, |x_2| \leq R/2$.

The following Harnack inequality in a disconnected region is the main theorem of this paper. See \Cref{sec-preliminaries} for the definition of the nonlocal tail, $\mathrm{Tail}(\cdot)$.
\begin{theorem}[Harnack inequality in a disconnected region]\label{thm-Harnack-disconnected}
    Let $u$ be a weak solution of \eqref{eq-main} in $B_{2r}(x_1) \cup B_{2r}(x_2)$ and $u$ is nonnegative in $B_R$. Suppose that either
    \begin{enumerate}[(i)]
        \item $s \in (0,1)$ and $k$ is translation invariant, i.e., there exists a function $K : \mathbb{R}^n \to [0, \infty]$ such that
        \begin{equation*}
            k(x, y)=K(x-y) \quad \text{for any $x, y \in \mathbb{R}^n$}; \quad \text{or}
        \end{equation*}
        
        \item $s \in (0, 1)$ with $2s<n$.
    \end{enumerate}    
    Then there exists a constant $C=C(n, s, \Lambda)>0$ such that
    \begin{equation*}
        \sup_{B_{r}(x_2)}u \leq C\inf_{B_r(x_1)}u+C\left(\frac{r}{R}\right)^{2s}\mathrm{Tail}(u_-; 0, R),
    \end{equation*}
    where $u_-\coloneqq\max\{-u, 0\}$ denotes the negative part of $u$.
\end{theorem}

As we observed in the earlier example under the local situation, \Cref{thm-Harnack-disconnected} is a purely nonlocal phenomenon: the nonlocal nature of operators ensures the transfer of information despite the region being disconnected. In particular, we show that \Cref{thm-Harnack-disconnected} is non-robust in the sense that the universal constant $C$ blows up when $s \to 1^-$; see \Cref{rmk-robust1} and \Cref{rmk-robust2} for details. Moreover, the validity of \Cref{thm-Harnack-disconnected} in nonlinear nonlocal equations, including $(-\Delta_p)^s$ for $p \in (1, \infty)$, can be an interesting problem.

We present two independent approaches to prove \Cref{thm-Harnack-disconnected} depending on the conditions (i) and (ii). The first approach relies on the construction of barriers and the application of the localized maximum principle, under the condition (i). We suggest two simple barrier functions $w_1$ and $w_2$ in \Cref{lem-barrier1} and \Cref{lem-barrier2} so that a certain linear combination of $w_1$ and $w_2$ becomes a weak subsolution of \eqref{eq-main} with an appropriate profile. The translation invariant condition in (i) is necessary only in \Cref{lem-barrier2}, where the finiteness of a certain integral term follows from the control of $\delta(w_2; x, y) \coloneqq 2w_2(x)-w_2(x+y)-w_2(x-y)$; see \cite{CS09} for instance. At this stage, the standard comparison principle (Korvenp\"a\"a--Kuusi--Palatucci~\cite[Lemma~6]{KKP17} for instance) between $u$ and barrier functions is not directly applicable, because we do not impose the nonnegativity of $u$ in the exterior domain $\mathbb{R}^n \setminus B_R$. Instead, we utilize a version of the localized maximum principle (\Cref{thm-loc-max-Om}) that is similar to Kim--Lee~\cite[Theorem~1.7]{KL25}.

In the second approach, we employ a representation formula for a solution $u$ of the Dirichlet problem, under the condition (ii). In particular, the solution $u$ can be represented as the Poisson integral of its exterior data with the associated Poisson kernel $P$. Then we can capture the long-range interaction of a solution $u$ between the two disjoint balls. While the fractional Laplacian case $(-\Delta)^s$ can be rather easily treated due to the explicit formula of the Poisson kernel, an additional effort is necessary to obtain the Poisson kernel estimates for general operators $\mathcal{L}$. For operators $\mathcal{L}$ without the translation invariant assumption, Kim--Weidner~\cite{KW24b} recently obtained the two-sided Poisson kernel estimates; see \Cref{cor-Green} for the precise statement. Although their theorem provides sharp estimates for $P$ in a very general situation, it requires an additional regularity assumption on $k$, in order to describe the boundary behavior of $P$. Moreover, since it deals with a general $C^{1,\alpha}$ domain $\Omega$, the dependence of a constant $c$ on $\Omega$ is not explicit.

We point out that interior estimates of $P$ in a ball $B_r(x_1)$ are sufficient for the purpose of this paper. Therefore, we develop a rather simple version of \Cref{cor-Green}; we derive the two-sided interior Poisson kernel estimates in $B_r(x_1)$ by tracking down the dependence of the constant $c$ on $B_r$, without the H\"older regularity condition on $k$. On the other hand, the condition (ii) $2s<n$ in \Cref{thm-Harnack-disconnected} seems quite essential in view of the Green function estimates by Kassmann--Kim--Lee~\cite{KKL23}. In fact, the behavior of the Green function is different depending on the relation between $n$ and $2s$. To the best of the author's knowledge, the corresponding result of \cite{KKL23} for the case $n \leq 2s$ is not available in the literature. We also refer to \cite{Buc16, CKS10, CS98, CS04, FRRO24, KM14} for the Green function and the Poisson kernel estimates in different settings.\newline

Let us end with an immediate corollary of \Cref{thm-Harnack-disconnected}: the Harnack inequality in an annulus, i.e., two disjoint intervals, when $n=1$. Such a consequence can be applied to extend the Liouville and B\^ocher theorems presented in \cite{KL25} for $n=1$. Note that for $n \geq 2$, one can iterate the Harnack inequality in a ball, together with the standard covering argument, to obtain the Harnack inequality in an annulus $B_R \setminus \{0\}$; see \Cref{thm-Harnack-annulus}. Nevertheless, the covering argument does not work in the special case $n=1$, since $(-R, R) \setminus \{0\}$ is no longer connected.

\begin{corollary}[Harnack inequality in disjoint intervals]\label{cor-Harnack-annulus}
   Let $n=1$ and let $u$ be a weak solution of \eqref{eq-main} in $(-R, R)\setminus\{0\}$ such that $u \geq 0$ in $(-R, R)$. Suppose that either
    \begin{enumerate}[(i)]
        \item $s \in (0, 1)$ and $k$ is translation invariant; or
        
        \item $s \in (0, 1/2)$.
    \end{enumerate}    
   Then for any $r \leq R/8$, we have
    \begin{equation*}
        \sup_{(-3r, -r)}u \leq C \inf_{(r, 3r)}u +C\left(\frac{r}{R}\right)^{2s} \mathrm{Tail}(u_-; 0, R),
    \end{equation*}
    where $C=C(s, \Lambda)>0$.
\end{corollary}

The paper is organized as follows. In \Cref{sec-preliminaries}, we collect several definitions and well-known results concerning function spaces and weak solutions of nonlocal equations.  \Cref{sec-harnack-barrier} and \Cref{sec-harnack-poisson} are devoted to the two different proofs of \Cref{thm-Harnack-disconnected} via the localized maximum principle with barriers and the representation formula with Poisson kernel estimates, respectively.

\section{Preliminaries}\label{sec-preliminaries}
We summarize several definitions of function spaces and weak solutions of nonlocal equations, and recall the standard Harnack inequality, which will be used throughout the paper. 

For $\Omega$ being a bounded $C^{1,1}$ domain in $\mathbb{R}^n$, the \emph{fractional Sobolev space} $W^{s, 2}(\Omega)$ consists of measurable functions $u: \Omega \to [-\infty, \infty]$ whose fractional Sobolev norm
\begin{align*}
\|u\|_{W^{s, 2}(\Omega)}
&\coloneqq \left( \|u\|_{L^2(\Omega)}^2 + [u]_{W^{s, 2}(\Omega)}^2 \right)^{1/2} \\
&\coloneqq \left( \int_\Omega |u(x)|^2 \,\mathrm{d}x + \int_\Omega \int_\Omega \frac{|u(x)-u(y)|^2}{|x-y|^{n+2s}} \,\mathrm{d}y\,\mathrm{d}x \right)^{1/2}
\end{align*}
is finite. By $W^{s, 2}_{\mathrm{loc}}(\Omega)$ we denote the space of functions $u$ such that $u \in W^{s, 2}(G)$ for every open $G \Subset \Omega$. We refer the reader to Di Nezza--Palatucci--Valdinoci~\cite{DNPV12} for further properties of these spaces.

Since we are concerned with nonlocal equations, we also need the \emph{tail space}
\begin{equation*}
L^{1}_{2s}(\mathbb{R}^n) \coloneqq \left\{ u~\text{measurable}: \int_{\mathbb{R}^n} \frac{|u(y)|}{(1+|y|)^{n+2s}} \,\mathrm{d}y < \infty \right\};
\end{equation*}
see Kassmann~\cite{Kas11} and Di Castro--Kuusi--Palatucci~\cite{DCKP16}. Note that $u \in L^{1}_{2s}(\mathbb{R}^n)$ if and only if the \emph{nonlocal tail} (or \emph{tail} for short)
\begin{equation}\label{eq-tail}
\mathrm{Tail}(u; x_0, r) \coloneqq r^{2s} \int_{\mathbb{R}^n \setminus B_r(x_0)} \frac{|u(y)|}{|y-x_0|^{n+2s}} \,\mathrm{d}y
\end{equation}
is finite for any $x_0 \in \mathbb{R}^n$ and $r>0$.

We next recall the definitions of a weak solution of \eqref{eq-main}, or more generally $\mathcal{L}u=f$, and of an $\mathcal{L}$-harmonic function. To this end, we define, for measurable functions $u, v: \mathbb{R}^n \to [-\infty, \infty]$, the bilinear form
\begin{equation*}
\mathcal{E}(u,v)\coloneqq\int_{\mathbb{R}^n} \int_{\mathbb{R}^n} (u(x)-u(y))(v(x)-v(y)) k(x, y) \,\mathrm{d}y\,\mathrm{d}x,
\end{equation*}
provided that the integral is finite.

\begin{definition}\label{def-harmonic}
Let $f \in L^{\infty}(\Omega)$. A function $u \in W^{s, 2}_{\mathrm{loc}}(\Omega) \cap L^{1}_{2s}(\mathbb{R}^n)$ is a {\it weak supersolution} (resp.\ \emph{weak subsolution}) of $\mathcal{L}u=f$ in $\Omega$ if
\begin{equation}\label{eq-sol}
\mathcal{E}(u, \varphi) \geq \int_\Omega f\varphi \,\mathrm{d}x ~\left(\text{resp.} \leq  \int_\Omega f\varphi \,\mathrm{d}x\right)
\end{equation}
for all nonnegative $\varphi \in C^\infty_c(\Omega)$. A function $u \in W^{s, 2}_{\mathrm{loc}}(\Omega) \cap L^{1}_{2s}(\mathbb{R}^n)$ is a {\it weak solution} of $\mathcal{L}u=f$ in $\Omega$ if \eqref{eq-sol} holds for all $\varphi \in C^\infty_c(\Omega)$. A function $u$ is \emph{$\mathcal{L}$-harmonic} in $\Omega$ if it is a weak solution of $\mathcal{L}u=0$ in $\Omega$ and $u \in C(\Omega)$. In particular, if $u$ is a continuous weak solution of $(-\Delta)^su=0$, then we say that $u$ is \emph{$s$-harmonic}.
\end{definition}

It is useful to deal with a larger class of test functions than $C^\infty_c(\Omega)$ as follows; see also \cite[Section~2]{KL25} and \cite[Section~2]{BBK24} for more comments on the test function classes.

\begin{proposition}[{\cite[Proposition~2.6]{KL24}}]\label{prop-test}
A function $u\in W^{s, 2}_{\mathrm{loc}}(\Omega) \cap L^{1}_{2s}(\mathbb{R}^n)$  is a weak supersolution (resp.\ weak subsolution) of \eqref{eq-main} in $\Omega$ if and only if \eqref{eq-sol} holds for all nonnegative $\varphi \in W^{s, 2}_{\mathrm{loc}}(\Omega)$ with $\supp{\varphi}\Subset \Omega$.
\end{proposition}

We end this section with Harnack inequalities for weak solutions of $\mathcal{L}u=0$ in a ball when $n \geq 1$ or in an annulus when $n \geq 2$; we again point out that both domains are clearly \emph{connected}.

\begin{theorem}[{\cite[Theorem~1.1]{DCKP14}}]\label{thm-Harnack}
Let $n \geq 1$. Let $u$ be a weak solution of \eqref{eq-main} in $B_R(x_0)$ such that $u \geq 0$ in $B_R(x_0)$. Then for any $r \leq R/2$,
	\begin{equation*}
		\sup_{B_r(x_0)}u \leq C \inf_{B_r(x_0)}u + C\mathrm{Tail}(u_-; x_0, r),
	\end{equation*}
where $C=C(n, s, \Lambda)>0$.
\end{theorem}

\begin{theorem}[{\cite[Theorem~3.9]{KL24}}]\label{thm-Harnack-annulus}
Let $n \geq 2$. Let $u$ be a weak solution of \eqref{eq-main} in $B_R \setminus \{0\}$ such that $u \geq 0$ in $B_R \setminus \{0\}$. Then for any $r \leq R/2$, 
\begin{equation*}
\sup_{B_r \setminus B_{r/2}}{u} \leq C \inf_{B_r \setminus B_{r/2}}{u} + C \left(\frac{r}{R}\right)^{2s}\mathrm{Tail}(u_-; 0, R),
\end{equation*}
where $C=C(n, s, \Lambda)>0$.
\end{theorem}

\section{Proof 1: localized maximum principle}\label{sec-harnack-barrier}
In this section, we construct barrier functions and employ the localized maximum principle introduced in \cite{KL25} to prove \Cref{thm-Harnack-disconnected} when $k$ is translation invariant. It is noteworthy that the extra translation invariant condition is necessary only in the construction of the second barrier function in \Cref{lem-barrier2}. 

We begin with the construction of two appropriate subsolutions of $\mathcal{L}u=0$ in $B_r(x_1)$. 

\begin{lemma}[Barrier 1]\label{lem-barrier1}
    Let $w_1=\chi_{B_r(x_2)}$. Then there exists a constant $C=C(n, s, \Lambda)>0$ such that
    \begin{equation*}
        \mathcal{L}w_1 \leq -Cr^{-2s} \quad \text{in $B_{r}(x_1)$}.
    \end{equation*}
\end{lemma}

\begin{proof}
    For $x \in B_r(x_1)$, the pointwise value $\mathcal{L}w_1(x)$ can be computed as follows:
    \begin{equation*}
        \begin{aligned}
            \mathcal{L}w_1(x)&=2\,\mathrm{p.v.}\int_{\mathbb{R}^n} (w_1(x)-w_1(y))k(x, y)\,\mathrm{d}y\\
            &=-2\int_{B_r(x_2)} k(x, y)\,\mathrm{d}y\\
            &\leq -2\Lambda^{-1}\int_{B_r(x_2)} |x-y|^{-n-2s}\,\mathrm{d}y.
        \end{aligned}
    \end{equation*}
    It follows from 
    \begin{equation*}
        |x-y| \leq |x-x_1|+|x_1-x_2|+|x_2-y| \leq 10r \quad \text{for $x \in B_r(x_1)$ and $y \in B_r(x_2)$}
    \end{equation*}
    that
    \begin{equation*}
        \mathcal{L}w_1(x) \leq -Cr^{-2s}
    \end{equation*}
    as desired.
\end{proof}

\begin{lemma}[Barrier 2]\label{lem-barrier2}
   Let $w_2 \in C_c^{\infty}(B_{r}(x_1))$ be a cut-off function such that $w_2 \equiv 1$ on $B_{r/2}(x_1)$, $0 \leq w_2 \leq 1$ in $\mathbb{R}^n$ and $\|D^2w_2\|_{\infty} \leq c(n)r^{-2}$.  Then there exists a constant $C=C(n, s, \Lambda)>0$ such that
    \begin{equation*}
        \mathcal{L}w_2 \leq Cr^{-2s} \quad \text{in $B_{r}(x_1)$}.
    \end{equation*}
\end{lemma}

\begin{proof}
    The proof is similar to the one of \Cref{lem-barrier1}. For $x \in B_r(x_1)$, we use the translation invariant property of $k$ as follows.
    \begin{equation*}
        \begin{aligned}
            \mathcal{L}w_2(x)&=2\,\mathrm{p.v.}\int_{\mathbb{R}^n} (w_2(x)-w_2(y))k(x,y)\,\mathrm{d}y\\
            &=\int_{\mathbb{R}^n} (2w_2(x)-w_2(x+y)-w_2(x-y))K(y) \,\mathrm{d}y\\
            &\leq \Lambda\int_{|y|<r} |2w_2(x)-w_2(x-y)-w_2(x+y)||y|^{-n-2s}  \,\mathrm{d}y+4\Lambda\int_{|y| \geq r} |y|^{-n-2s}  \,\mathrm{d}y\\
            &\leq C \|D^2w_2\|_{\infty} \int_{|y|<r} |y|^{-n-2s+2}  \,\mathrm{d}y+Cr^{-2s}\\
            &\leq Cr^{-2s}. 
        \end{aligned}
    \end{equation*}
\end{proof}

In order to compare a weak solution $u$ with barriers, we need the localized maximum principle that is a modified version of \cite[Theorem~1.7]{KL25}. For this purpose, we show the following lemma, which captures the effect of a non-homogeneous term in the maximum principle for nonlocal operators. 

\begin{lemma}\label{lem-MP-nonhom}
Let $C_0$ be a nonnegative constant. Let $u$ be a weak supersolution of $\mathcal{L}u = -C_0$ in $B_{r}(x_1)$ such that $u \geq 0$ a.e.\ in $\mathbb{R}^n \setminus B_r(x_1)$. Then there exists a constant $C=C(n, s, \Lambda)>0$ such that
\begin{equation*}
    u \geq -CC_0r^{2s} \quad \text{a.e.\ in $B_r(x_1)$}.
\end{equation*}
\end{lemma}

\begin{proof}
It is enough to prove that the set
\begin{equation*}
G\coloneqq\{u<-CC_0r^{2s}\} \, (\subset B_r(x_1))
\end{equation*}
has measure zero for a sufficiently large constant $C>0$. We define the function
\begin{equation*}
\varphi=(u+CC_0r^{2s})_-,
\end{equation*}
which is admissible for a test function in the weak formulation of $\mathcal{L}u=-C_0$; see \Cref{prop-test} or \cite[Proposition~2.6]{KL25} for instance. We observe that 
\begin{enumerate}[(i)]
    \item $G^c \supset \mathbb{R}^n \setminus B_r(x_1) \supset \mathbb{R}^n \setminus B_{2r}(x)$ for any $x \in B_r(x_1)$;

    \item $u(x)-u(y) \leq 0$ for $x \in G$ and $y \in G^c$.
\end{enumerate}
Thus, we have
\begin{align*}
-C_0 \int_G \varphi \,\mathrm{d}x
&\leq \mathcal{E}(u, \varphi) \\
&\leq - \Lambda^{-1}\int_G \int_G \frac{|u(x)-u(y)|^2}{|x-y|^{n+2s}} \,\mathrm{d}y \,\mathrm{d}x + 2\Lambda^{-1}\int_G \int_{G^c} \frac{(u(x)-u(y))\varphi(x)}{|x-y|^{n+2s}} \,\mathrm{d}y \,\mathrm{d}x \\
&\leq 2\Lambda^{-1}\int_G \int_{\mathbb{R}^n \setminus B_{2r}(x)} \frac{(u(x)-u(y))\varphi(x)}{|x-y|^{n+2s}} \,\mathrm{d}y \,\mathrm{d}x \\
&\leq -2CC_0 \Lambda^{-1} r^{2s}\int_G \varphi(x) \int_{\mathbb{R}^n \setminus B_{2r}(x)} \frac{1}{|x-y|^{n+2s}} \,\mathrm{d}y \,\mathrm{d}x \\
&= - \frac{CC_0 |\mathbb{S}^{n-1}|}{s \Lambda 2^{2s}} \int_G \varphi \,\mathrm{d}x.
\end{align*}
Taking $C>2^{2s}s\Lambda /|\mathbb{S}^{n-1}|$ yields that $|G|=0$, which finishes the proof.
\end{proof}

We next prove the localized maximum principle that is suitable for our purpose. As pointed out in \cite{KL25}, see Lindgren--Lindqvist~\cite[Lemma~9]{LL14}, Korvenp\"a\"a--Kuusi--Palatucci~\cite[Lemma~6]{KKP17} or Kim--Lee~\cite[Lemma~5.2]{KL23} for the standard maximum principle with respect to nonlocal operators.

\begin{lemma}[Localized maximum principle]\label{thm-loc-max-Om}
Let $u$ be a weak supersolution of \eqref{eq-main} in $B_r(x_1)$ such that $u \geq 0$ in a.e.\ $B_{R} \setminus B_{r}(x_1)$. Then there exists a constant $C=C(n,s, \Lambda)>0$ such that 
\begin{equation*}
u \geq -C\left(\frac{r}{R}\right)^{2s}\mathrm{Tail}(u_-; 0, R) \quad \text{a.e.\ in $B_r(x_1)$}.
\end{equation*}
\end{lemma}

\begin{proof}
Let $\eta \in C^\infty_c(B_{R})$ be a cut-off function such that $0 \leq \eta \leq 1$ in $\mathbb{R}^n$ and $\eta \equiv 1$ on $B_{R/2}$. Then for any nonnegative function $\varphi \in C^\infty_c(B_r(x_1))$, we have
\begin{align*}
\mathcal{E}(u\eta, \varphi)
&\geq - \mathcal{E}(u(1-\eta), \varphi) \\
&= - \int_{B_r(x_1)} \int_{B_r(x_1)} (u(x)(1-\eta(x)) - u(y)(1-\eta(y)))(\varphi(x) -\varphi(y)) k(x, y)\,\mathrm{d}y \,\mathrm{d}x \\
&\quad - 2 \int_{B_r(x_1)} \int_{\mathbb{R}^n \setminus B_r(x_1)} (u(x)(1-\eta(x)) - u(y)(1-\eta(y))) \varphi(x) k(x, y) \,\mathrm{d}y \,\mathrm{d}x \\
&= 2 \int_{B_r(x_1)} \int_{\mathbb{R}^n \setminus B_{R/2}} u(y)(1-\eta(y))\varphi(x) k(x, y) \,\mathrm{d}y \,\mathrm{d}x \\
&\geq -2 \int_{B_r(x_1)} \int_{\mathbb{R}^n \setminus B_{R}} u_-(y) \varphi(x) k(x, y) \,\mathrm{d}y \,\mathrm{d}x \\
&\geq -2\Lambda C_0 \int_{B_r(x_1)} \varphi \,\mathrm{d}x,
\end{align*}
where
\begin{equation*}
C_0 \coloneqq \sup_{x \in B_r(x_1)} \int_{\mathbb{R}^n \setminus B_{R}} \frac{u_-(y)}{|x-y|^{n+2s}} \,\mathrm{d}y.
\end{equation*}
Since
\begin{equation*}
|x-y| \geq |y| - |x-x_1|-|x_1| \geq |y|/2 \quad \text{for $x \in B_r(x_1)$ and $y \in \mathbb{R}^n \setminus B_{R}$},
\end{equation*}
 we have
\begin{align*}
0 \leq C_0
\leq  2^{n+2s} \int_{\mathbb{R}^n \setminus B_{R}} \frac{u_-(y)}{|y|^{n+2s}} \,\mathrm{d}y=2^{n+2s} R^{-2s}\mathrm{Tail}(u_-; 0, R).
\end{align*}
Hence, we conclude that $u\eta$ is a weak supersolution of $\mathcal{L}(u\eta) = -2\Lambda C_0$ in $B_r(x_1)$.

On the other hand, since $u\eta \geq 0$ in $\mathbb{R}^n \setminus B_r(x_1)$, we are able to apply \Cref{lem-MP-nonhom} to find that
\begin{equation*}
    u=u\eta \geq -C\left(\frac{r}{R}\right)^{2s}\mathrm{Tail}(u_-; 0, R) \quad \text{a.e.\ in $B_r(x_1)$}.
\end{equation*}
for some constant $C=C(n, s, \Lambda)>0$.
\end{proof}

We are now ready to prove \Cref{thm-Harnack-disconnected} via the localized maximum principle with the barrier functions, provided that $k$ is translation invariant.

\begin{proof}[Proof of \Cref{thm-Harnack-disconnected} (i)]
    We set $v \coloneqq w_1+c_0w_2$ for barrier functions $w_1$ and $w_2$ constructed in \Cref{lem-barrier1} and \Cref{lem-barrier2}, respectively, where $c_0$ is a universal positive constant that will be determined soon. In fact, due to \Cref{lem-barrier1} and \Cref{lem-barrier2}, there exists a positive constant $c_0$ depending only on $n$, $s$ and $\Lambda$ such that
    \begin{equation*}
        \left\{ \begin{aligned}
        \mathcal{L}v &\leq 0 && \text{in $B_r(x_1)$} \\
        v&= c_0 && \text{on $B_{r/2}(x_1)$}\\
        v&= 1 && \text{on $B_{r}(x_2)$}\\
        v&=0 && \text{on $B_R \setminus (B_{r}(x_1)\cup B_r(x_2))$}\\
        v&=0 && \text{on $\mathbb{R}^n \setminus B_R$}.
        \end{aligned} \right.
    \end{equation*}

    We now apply the localized maximum principle (\Cref{thm-loc-max-Om}) for $u-mv$ with $m \coloneqq  \inf_{B_r(x_2)}u$ to obtain that
    \begin{equation*}
        \begin{aligned}
            u &\geq mv -C\left(\frac{r}{R}\right)^{2s}\mathrm{Tail}((u-mv)_-; 0, R)\\
                &\geq c_0m -C\left(\frac{r}{R}\right)^{2s}\mathrm{Tail}(u_-; 0, R) \quad \text{in $B_{r/2}(x_1)$}.
        \end{aligned}
    \end{equation*}
    In particular, we arrive at
    \begin{equation*}
        C \inf_{B_r(x_2)}u \leq u(x_1)+C\left(\frac{r}{R}\right)^{2s}\mathrm{Tail}(u_-; 0, R).
    \end{equation*}
    We finish the proof by applying the standard Harnack inequality (\Cref{thm-Harnack}) in each ball $B_{2r}(x_1)$ and $B_{2r}(x_2)$, together with the fact that $u \geq 0$ in $B_R$.
\end{proof}

\begin{remark}[Non-robustness 1]\label{rmk-robust1}
    It is natural to ask whether the universal constant $C>0$ appearing in \Cref{thm-Harnack-disconnected} is robust or not. More precisely, we would like to find the explicit dependence of $C$ on the parameter $s \in (0,1)$ and send $s \to 1^{-}$. In fact, we expect that $C$ must blow up when we let $s \to 1^-$, in view of the simple observation in \Cref{sec-introduction} for the local setting.
    
    For simplicity, we may assume that $u \geq 0$ in $\mathbb{R}^n$. To verify the non-robustness, we first have to specify the dependence of the kernel $k$ on $s$ as follows:
    \begin{equation*}
        \Lambda^{-1}(1-s) |x-y|^{-n-2s} \leq k(x, y)=k(y, x) \leq \Lambda (1-s)|x-y|^{-n-2s}.
    \end{equation*}
    Then it is easy to check that two barriers $w_1$ and $w_2$ satisfy
    \begin{equation*}
        \mathcal{L}w_1 \leq -C(1-s)r^{-2s} \quad \text{and} \quad \mathcal{L}w_2 \leq Cr^{-2s} \quad \text{in $B_r(x_1)$},
    \end{equation*}
    where $C>0$ depends only on $n$ and $\Lambda$ but not on $s$. In fact, in the proof of \Cref{lem-barrier2}, the estimate of the integral $\int_{|y|<r}|y|^{-n-2s+2}\,\mathrm{d}y$ requires us to multiply an additional $1/(1-s)$. Since we assume that $u$ is nonnegative in $\mathbb{R}^n$, we can just apply the global maximum principle to obtain that the constant $c_0=c_0(n, \Lambda)$ is comparable to $1-s$. Hence, we conclude that the inequality gives no information when we let $s \to 1^-$, as we expected.
\end{remark}

\section{Proof 2: Poisson kernel estimates}\label{sec-harnack-poisson}
\subsection{Poisson kernel estimates}\label{sec-poisson}
In this subsection, we display the representation formula for weak solutions of nonlocal linear Dirichlet problems in terms of the Poisson kernel $P$, and derive the two-sided estimates for $P$.

Given a bounded $C^{1,1}$ domain $\Omega \subset \mathbb{R}^n$ and an exterior datum $g \in L^1_{2s}(\mathbb{R}^n \setminus \Omega)$, a unique solution $u$ of the Dirichlet problem 
\begin{equation*}
        \begin{cases}
            \mathcal{L}u=0 & \text{in $\Omega$}\\
            u=g & \text{on $\mathbb{R}^n \setminus \Omega$}
        \end{cases}
    \end{equation*}
    can be written by the representation formula
     \begin{equation*}
        u(x)=\int_{\mathbb{R}^n \setminus \Omega} g(z)P(x, z)\,\mathrm{d}z \quad \text{for all $x \in \Omega$},
    \end{equation*}
    where $P=P_{\mathcal{L}, \Omega}: \Omega \times (\mathbb{R}^n \setminus \Omega) \to [0, \infty]$ denotes the \emph{Poisson kernel} associated with $\mathcal{L}$ and $\Omega$. Moreover, the nonlocal Gauss--Green formula explains the relation between the Poisson kernel $P$ and the Green function $G$ 
    \begin{equation*}
        P(x, z)=\int_{\Omega}G(x, y)k(z, y)\,\mathrm{d}y \quad \text{for any $x \in \Omega$ and $z \in \mathbb{R}^n \setminus \Omega$}.
    \end{equation*}
    We refer to Bucur~\cite{Buc16}, Chen--Song~\cite{CS98} and Kim--Weidner~\cite{KW24b} for the definition and further properties of the Green function.

    We now move our attention to the two-sided estimates for the Poisson kernel $P$. Chen--Song~\cite[Theorem~1.5]{CS98} obtained upper and lower bounds for $G$ and $P$ when $n \geq 2$, the kernel $k$ is translation invariant and $\Omega$ is a $C^{1, 1}$ domain. Recently, Kim--Weidner~\cite[Corollary~9.6]{KW24b} developed a PDE approach to derive the two-sided estimates for $P$ in a more general setting, i.e., when $2s<n$, $k$ is H\"older regular in the sense of \eqref{eq-kernel-holder} and $\Omega$ is a $C^{1, \alpha}$ domain.

    \begin{theorem}[{\cite[Corollary~9.6]{KW24b}}]\label{cor-Green}
    Let $s \in (0,1)$ with $2s<n$ and $\alpha, \sigma \in (0,s)$. Let $\Omega \subset \mathbb{R}^n$ be a $C^{1, \alpha}$ domain. Moreover, suppose that the kernel $k$ is locally H\"older continuous of order $\sigma$ in the following sense: 
    \begin{equation}\label{eq-kernel-holder}
        |k(x+h, y+h)-k(x, y)| \leq \Lambda \frac{|h|^{\sigma}}{|x-y|^{n+2s}}\quad \text{for all $x, y \in \Omega$ and $h \in B_1$.}
    \end{equation}
    Then it holds for any $x \in \Omega$ and $z \in \mathbb{R}^n \setminus \Omega$,
         \begin{equation*}
        c^{-1} \frac{d^s(x)}{d^s(z)(1+d(z))^s} \frac{1}{|x-z|^{n}}\leq P(x, z) \leq c \frac{d^s(x)}{d^s(z)(1+d(z))^s}\frac{1}{|x-z|^{n}}
     \end{equation*}
     for some universal constant $c=c(n,s, \Lambda, \Omega) \geq 1$, where $d(y) \coloneqq \mathrm{dist}(y, \partial \Omega)$ is the Euclidean distance between $y$ and $\partial \Omega$. 
    \end{theorem}

   As we mentioned in \Cref{sec-introduction}, we now propose a simplified version of \Cref{cor-Green}, because we are concerned with the interior estimate of $P$ associated with a ball. For this purpose, we first recall the bounds near the diagonal for the Green function $G$ developed in \cite{KKL23}. In fact, Kassmann--Kim--Lee~\cite{KKL23} proved \Cref{thm-Greenest} when $n \geq 3$ to guarantee the robustness of the comparable constant $C>0$. We note that the same (but not robust) estimates can still be obtained from the same arguments under the weaker condition $2s<n$.
    
     \begin{theorem}[Bounds for Green function, {\cite[Theorems~1.4 and 1.5]{KKL23}}]\label{thm-Greenest}
         Let $s \in (0, 1)$ with $2s<n$. Then the Green function $G$ of $\mathcal{L}$ on $\Omega$ satisfies
         \begin{equation*}
             G(x, y) \leq C|x-y|^{2s-n} \quad \text{for all $x, y \in \Omega$}
         \end{equation*}
         and
          \begin{equation*}
             G(x, y) \geq C^{-1}|x-y|^{2s-n} \quad \text{for all $x, y \in \Omega$ with $|x-y| \leq \min\{d(x), d(y)\}$}
         \end{equation*}
         for some constant $C>0$ depending only on $n$, $s$ and $\Lambda$, but not on $\Omega$.
     \end{theorem}

    We use the interior estimates for $G$ to develop the interior estimates for the Poisson kernel $P$ associated with a ball $B_r(x_1)$.
    \begin{theorem}[Bounds for Poisson kernel]\label{thm-poissonest}
        Let $s \in (0, 1)$ with $2s<n$. Let $P$ be the Poisson kernel associated with $\mathcal{L}$ and $B_r(x_1)$. Then it holds for any $x \in B_{r/2}(x_1)$ and $z \in \mathbb{R}^n \setminus B_{2r}(x_1)$,
         \begin{equation*}
           C^{-1}\frac{r^{2s}}{|x-z|^{n+2s}} \leq P(x, z) \leq C\frac{r^{2s}}{|x-z|^{n+2s}}
         \end{equation*}
         for some constant $C>0$ depending only on $n$, $\Lambda$ and $s$, but not on $r$ and $x_1$. 
    \end{theorem}

    \begin{proof}
        We mainly follow the argument in \cite[Section~3]{CS98}. Let us recall the relation between $P$ and $G$ as follows:
        \begin{equation*}
        P(x, z)=\int_{B_r(x_1)}G(x, y)k(z, y)\,\mathrm{d}y \quad \text{for any $x \in B_r(x_1)$ and $z \in \mathbb{R}^n \setminus B_r(x_1)$}.
    \end{equation*}
    Since we are interested in the interior estimates of $P$, we set $x \in B_{r/2}(x_1)$ and $z \in \mathbb{R}^n \setminus B_{2r}(x_1)$.\newline

    (i) (Upper bound) 
     Since
     \begin{equation*}
        |z-x|\leq |z-y|+|y-x| \leq |z-y|+3r/2 \leq 5|z-y|/2 
    \end{equation*}
    for $x \in B_{r/2}(x_1)$, $y \in B_r(x_1)$ and $z \in \mathbb{R}^n \setminus B_{2r}(x_1)$, we observe that
    \begin{equation*}
        \begin{aligned}
            P(x, z) &\leq \Lambda\int_{B_r(x_1)}\frac{G(x, y)}{|y-z|^{n+2s}}\,\mathrm{d}y \\
            &\leq \Lambda\frac{2^{n+2s}}{5^{n+2s}|x-z|^{n+2s}}\int_{B_r(x_1)}G(x, y)\,\mathrm{d}y.
        \end{aligned}
    \end{equation*}
    Since \Cref{thm-Greenest} shows that
    \begin{equation*}
        G(x, y) \leq C|x-y|^{2s-n}  \quad \text{for $x, y \in B_r(x_1)$},
    \end{equation*}
    we have
    \begin{equation*}
        \begin{aligned}
            \int_{B_r(x_1)} G(x, y) \,\mathrm{d}y &\leq C\int_{B_r(x_1)}\frac{1}{|x-y|^{n-2s}} \,\mathrm{d}y\\
            &\leq C\int_{B_{2r}(x)}\frac{1}{|x-y|^{n-2s}} \,\mathrm{d}y\\
            &\leq Cr^{2s}.
        \end{aligned}
    \end{equation*}

    (ii) (Lower bound) Since
      \begin{equation*}
        |y-z| \leq |x-z|+|x-y| \leq 5|x-z|/4 
    \end{equation*}
    for $y \in B_{r/4}(x)$ and $z \in \mathbb{R}^n \setminus B_{2r}(x_1)$, we observe that
    \begin{equation*}
        \begin{aligned}
            P(x, z) &\geq \Lambda^{-1}\int_{B_r(x_1)}\frac{G(x, y)}{|y-z|^{n+2s}}\,\mathrm{d}y \\
            &\geq \Lambda^{-1}\int_{B_{r/4}(x)}\frac{G(x, y)}{|y-z|^{n+2s}}\,\mathrm{d}y\\
            &\geq \Lambda^{-1} \frac{4^{n+2s}}{5^{n+2s}}\frac{1}{|x-z|^{n+2s}} \int_{B_{r/4}(x)}G(x, y)\,\mathrm{d}y.
        \end{aligned}
    \end{equation*}
    Note that
    \begin{equation*}
        |x-y| \leq r/4 \leq \min\{d(x), d(y)\} \quad \text{for $x\in B_{r/2}(x_1)$ and $y \in B_{r/4}(x)$}.
    \end{equation*}
    It follows from \Cref{thm-Greenest} that
    \begin{equation*}
        G(x, y) \geq C^{-1}|x-y|^{2s-n} \quad \text{for $x\in B_{r/2}(x_1)$ and $y \in B_{r/4}(x)$}
    \end{equation*}
    and so 
    \begin{equation*}
        \begin{aligned}
            \int_{B_{r/4}(x)} G(x, y) \,\mathrm{d}y &\geq C^{-1}\int_{B_{r/4}(x)}\frac{1}{|x-y|^{n-2s}} \,\mathrm{d}y\\
            &\geq C^{-1}r^{2s}. 
        \end{aligned}
    \end{equation*}
    \end{proof}

\begin{remark}[Poisson kernel for $(-\Delta)^s$]\label{rmk-s-harmonic}
    When it comes to $s$-harmonic functions, we do not need the assumption $2s<n$ since there exists an explicit formula for the Poisson kernel in this special case. In particular, if we let $P^s(x, y)$ be the Poisson kernel associated with $(-\Delta)^s$ and a ball $B_r=B_r(0)$, then it can be written as
    \begin{equation*}
        P^s(x, y) \coloneqq
            c_{n,s} \frac{(r^2-|x|^2)^s}{(|y|^2-r^2)^s} \frac{1}{|y-x|^n},
    \end{equation*}
    where $s \in (0,1)$ and $c_{n,s}=\Gamma(n/2)\pi^{-n/2-1}\sin(s\pi)$; see \cite{BGR61, BKN02} for details. 
\end{remark}

\subsection{Proof of \Cref{thm-Harnack-disconnected}}
The goal of this subsection is to prove \Cref{thm-Harnack-disconnected} with the aid of the two-sided Poisson kernel estimates obtained in \Cref{sec-poisson}. As mentioned in \Cref{sec-introduction}, we require an assumption $2s<n$ in this approach, since the Poisson kernel estimates strongly rely on the associated Green function estimates in the context of \cite{KLL23, KW24b}. 

Let us begin with a weak Harnack inequality for weak supersolutions of \eqref{eq-main} in two disjoint balls. In contrast to \Cref{thm-Harnack-disconnected}, the lemma below requires $u$ to be a weak supersolution only in $B_{2r}(x_1)$ and imposes no condition in $B_{2r}(x_2)$.

\begin{lemma}[Weak Harnack inequality]\label{lem-WHI}
Let $s \in (0, 1)$ with $2s<n$. Let $u$ be a weak supersolution of \eqref{eq-main} in $B_{2r}(x_1)$ and $u$ is nonnegative in $B_{R}$. Then there exists a constant $C=C(n, s, \Lambda)>0$ such that
     \begin{equation}\label{eq-WHI}
        C \fint_{B_r(x_2)}u(z)\,\mathrm{d}z \leq C\inf_{B_r(x_1)}u+C\left(\frac{r}{R}\right)^{2s}\mathrm{Tail}(u_-; 0, R).
    \end{equation}
\end{lemma}

\begin{proof} 
 We choose a point $\bar x \in \overline B_r(x_1)$ such that 
\begin{equation*}
    u(\bar x)=\inf_{B_r(x_1)} u.
\end{equation*}

We let $P : B_{3r/2}(x_1) \times (\mathbb{R}^n \setminus B_{3r/2}(x_1)) \to [0, \infty]$ be the Poisson kernel associated with $\mathcal{L}$ and $B_{3r/2}(x_1)$. If we let $v$ be the unique solution of the Dirichlet problem
    \begin{equation*}
        \begin{cases}
            \mathcal{L}v=0 & \text{in $B_{3r/2}(x_1)$}\\
            v=u & \text{on $\mathbb{R}^n \setminus B_{3r/2}(x_1)$},
        \end{cases}
    \end{equation*}
    then \Cref{sec-poisson} shows that the following representation formula holds:
    \begin{equation*}
        v(x)=\int_{\mathbb{R}^n \setminus B_{3r/2}(x_1)} u(z)P(x, z)\,\mathrm{d}z \quad \text{for all $x \in B_{3r/2}(x_1)$}.
    \end{equation*}
    On the other hand, the comparison principle between $u$ and $v$ in $B_{3r/2}(x_1)$ yields that $u \geq v$ in $B_{3r/2}(x_1)$; see \cite[Lemma~6]{KKP16} for instance. In particular, the nonnegativity of $u$ in $B_R$ gives us that
     \begin{equation*}
        \begin{aligned}
             u(\bar x)\geq v(\bar x)&=\int_{\mathbb{R}^n \setminus B_{3r/2}(x_1)} u(z)P(\bar x, z)\,\mathrm{d}z\\
             &\geq \int_{B_{r}(x_2)} u(z)P(\bar x, z)\,\mathrm{d}z+\int_{\mathbb{R}^n \setminus B_R} u(z)P(\bar x, z)\,\mathrm{d}z\\
             &\geq C^{-1}\int_{B_r(x_2)}  \frac{r^{2s}}{|\bar x-z|^{n+2s}} u(z) \,\mathrm{d}z-C\int_{\mathbb{R}^n \setminus B_R}\frac{r^{2s}}{|\bar x-z|^{n+2s}} u_-(z) \,\mathrm{d}z\\
             &\eqqcolon I_1+I_2,
        \end{aligned}
    \end{equation*}
    where we used the two-sided Poisson kernel estimates (\Cref{thm-poissonest})
    \begin{equation*}
        C^{-1} \frac{r^{2s}}{|\bar x-z|^{n+2s}}\leq P(\bar x, z) \leq C\frac{r^{2s}}{|\bar x-z|^{n+2s}} 
     \end{equation*}
     for $\bar x \in B_r(x_1)$ and $z \in B_r(x_2) \cup (\mathbb{R}^n \setminus B_R)$.
    
    We now estimate two integral terms $I_1$ and $I_2$ as follows:
    \begin{enumerate}[(i)]
    \item  For $z \in B_r(x_2)$, we have 
    \begin{equation*}
        |z-\bar x| \leq |z-x_2|+|x_2-x_1|+|x_1-\bar x| \leq 10r
    \end{equation*}
    and so 
    \begin{equation*}
        I_1 \geq C \fint_{B_r(x_2)}u(z)\,\mathrm{d}z.
    \end{equation*}
    
    \item For $z \in \mathbb{R}^n \setminus B_R$, we have 
    \begin{equation*}
        |z-\bar x| \geq |z|-|\bar x| \geq \frac{|z|}{2}
    \end{equation*}
    and so 
    \begin{equation*}
       I_2 \geq -Cr^{2s}\int_{\mathbb{R}^n \setminus B_R} \frac{u_-(z)}{|z|^{n+2s}}\,\mathrm{d}z=-C\left(\frac{r}{R}\right)^{2s}\mathrm{Tail}(u_-; 0, R).
    \end{equation*}
    \end{enumerate}

    By combining these estimates above, we conclude that
    \begin{equation*}
       C \fint_{B_r(x_2)}u(z)\,\mathrm{d}z \leq C\inf_{B_r(x_1)}u+C\left(\frac{r}{R}\right)^{2s}\mathrm{Tail}(u_-; 0, R).
    \end{equation*}
\end{proof}

We are now ready to prove \Cref{thm-Harnack-disconnected} via the Poisson kernel estimates, when $2s<n$. 
\begin{proof}[Proof of \Cref{thm-Harnack-disconnected} (ii)]
   In view of \Cref{lem-WHI}, we focus on the estimate of the left-hand side of \eqref{eq-WHI} by employing the standard Harnack inequality. In fact, we apply the Harnack inequality (\Cref{thm-Harnack}) in a ball $B_{2r}(x_2)$ to find that
    \begin{equation*}
        \sup_{B_r(x_2)}u \leq C\inf_{B_r(x_2)}u+C\mathrm{Tail}(u_-; x_2, r).
    \end{equation*}
    Here we observe that
    \begin{equation*}
        \mathrm{Tail}(u_-; x_2, r)=r^{2s}\int_{\mathbb{R}^n \setminus B_R} \frac{u_-(y)}{|y-x_2|^{1+2s}} \leq C\left(\frac{r}{R}\right)^{2s}\mathrm{Tail}(u_-; 0, R),
    \end{equation*}
    since
    \begin{equation*}
        |y-x_2| \geq |y|-|x_2| \geq \frac{|y|}{2}  \quad \text{for all $y \in \mathbb{R}^n \setminus B_R$}.
    \end{equation*}
    Therefore, we obtain that
    \begin{equation}\label{eq-harnack1}
        \begin{aligned}
             \sup_{B_r(x_2)}u &\leq C\inf_{B_r(x_2)}u+C\left(\frac{r}{R}\right)^{2s}\mathrm{Tail}(u_-; 0, R)\\
             &\leq C\fint_{B_r(x_2)}u(z)\,\mathrm{d}z+C\left(\frac{r}{R}\right)^{2s}\mathrm{Tail}(u_-; 0, R)
        \end{aligned}
    \end{equation}
    The desired estimate follows from the combination of \eqref{eq-WHI} and \eqref{eq-harnack1}.
\end{proof}

\begin{remark}
    In view of the explicit Poisson kernel expression in \Cref{rmk-s-harmonic}, the second approach for \Cref{thm-Harnack-disconnected} indeed works for $(-\Delta)^s$ without the condition $2s<n$. Note that this fractional Laplacian case was already covered in the first approach, since $(-\Delta)^s$ is clearly translation invariant.    
\end{remark}

\begin{remark}[Non-robustness 2]\label{rmk-robust2}
    As we already checked in \Cref{rmk-robust1}, the robustness does not hold in \Cref{thm-Harnack-disconnected}. In the approach based on the Poisson kernel, we can observe that such a non-robustness arises from the following Gauss--Green formula:
    \begin{equation*}
        P(x, z)=\int_{\Omega}G(x, y)k(z, y)\,\mathrm{d}y \quad \text{for any $x \in \Omega$ and $z \in \mathbb{R}^n \setminus \Omega$}.
    \end{equation*}
    Here we again point out that $k \sim (1-s)$ and so $P \sim (1-s)$. Hence, the representation formula for $u$ yields that the constant $C>0$ in \Cref{thm-Harnack-disconnected} must blow up.
\end{remark}

\subsection*{Acknowledgements}
The author thanks Minhyun Kim for valuable discussions regarding the Poisson kernel estimates in \Cref{sec-poisson} and the non-robustness in \Cref{rmk-robust2}.

Se-Chan Lee is supported by the KIAS Individual Grant (No. MG099001) at the Korea Institute for Advanced Study.

\bibliographystyle{abbrv}
\bibliography{literature}

\end{document}